\documentclass[12pt]{amsart}

\usepackage{amsmath, amsthm, amssymb, amsfonts, amscd, graphicx, endnotes, tikz, color}

\usepackage{mathtools}

\usepackage[margin=01.2in]{geometry}

\usetikzlibrary{arrows}

\def\FF{{\mathbb F}}

\def\QQ{{\mathbb Q}}

\def\QQ{{\mathbb Q}}
\def\RR{{\mathbb R}}

\def\ZZ{{\mathbb Z}}

\def\0{{\mathbf 0}}
\def\1{{\mathbf 1}}

\def\s{{\mathbf s}}
\def\t{{\mathbf t}}

\def\Acal{{\mathcal A}}
\def\Bcal{{\mathcal B}}
\def\Ccal{{\mathcal C}}

\def\Hcal{{\mathcal H}}

\def\Mcal{{\mathcal M}}

\def\diam{\mathrm{diam}}

\def\min{\mathrm{min}}

\def\Char{\mathrm{char}}

\def\sup{\mathrm{sup}}
\def\max{\mathrm{max}}

\theoremstyle{plain}

\newtheorem{thm}{Theorem}

\newtheorem{lem}[thm]{Lemma}

\theoremstyle{definition}

\title[A family of hyperbolic $p$-adic plane automorphisms]{Attractors associated to a family of \\ hyperbolic $p$-adic plane automorphisms}

\author{Clayton Petsche}

\address{Clayton Petsche; Department of Mathematics; Oregon State University; Corvallis OR 97331 U.S.A.}

\email{petschec@math.oregonstate.edu}


\begin{document}

\begin{abstract}
We consider a certain two-parameter family of automorphisms of the affine plane over a complete, locally compact non-Archimedean field.  Each of these automorphisms admits a chaotic attractor on which it is topologically conjugate to a full two-sided shift map, and the attractor supports a unit Borel measure which describes the distribution of the forward orbit of Haar-almost all points in the basin of attraction.  We also compute the Hausdorff dimension of the attractor, which is non-integral. 
\end{abstract}

\maketitle

%%%%%%%%%%%%%%%%%
%%%%%%%%%%%%%%%%%
%%%%%%%%%%%%%%%%%
%%%%%%%%%%%%%%%%%
%%%%%%%%%%%%%%%%%
%%%%%%%%%%%%%%%%%

\section{Introduction}

Let $K$ be a complete and locally compact field with respect to a nontrivial, non-Archimedean absolute value $|\cdot|$.  Let 
$$
R=\{x\in K\mid |x|\leq1\}
$$ 
be the ring of integers in $K$, and let $\pi\in R$ be a uniformizing parameter; thus $|\pi|$ is maximal among all $x\in R$ with $|x|<1$, and
$$
\pi R=\{x\in R\mid |x|<1\}
$$
is the unique maximal ideal of $R$.  The most well-known examples are the fields $K=\QQ_p$ of $p$-adic numbers for prime numbers $p$, or more generally any finite extensions of $\QQ_p$.

Let $\FF=R/\pi R$ denote the residue field of $K$, and let $q=|\FF|$ denote its order.  We select the normalization of the absolute value $|\cdot|$ for which $|\pi|=1/q$; this choice is not strictly necessary but it simplifies many of our calculations.  

In \cite{MR3757169}, Allen-DeMark-Petsche studied the H\'enon map $H:K^2\to K^2$ given by $H(x,y) = (a+by-x^2,x)$ in the case $\Char(\FF)\neq2$.  For certain parameters $a,b\in K$, they found that $H$ is topologically conjugate on its filled Julia set to the shift map on the space of bisequences on two symbols, and for certain other choices of parameters, $H$ admits an attractor which appears in many cases to be uncountably infinite.  In \cite{MR3757169} $\S$~4.3, working with a particular family of examples over $\QQ_3$, they proved that the attractor is infinite and supports a measure which equidistributes the forward orbits of all points in the basin of attraction.  These attractors are thus similar in some ways to the strange attractors associated to certain real H\'enon maps; see \cite{MR0422932}, \cite{MR1087346}.  On the other hand, unlike the real H\'enon attractors, which are chaotic and have nonintegral Hausdorff dimension, the H\'enon maps studied in \cite{MR3757169} $\S$~4 are non-expanding and hence nonchaotic on their attractors, and the $3$-adic attractor described in \cite{MR3757169} $\S$~4.3 has Hausdorff dimension $1$; so the similarities to real strange attractors are limited.

In the present paper we study a family of hyperbolic non-Archimedean plane automorphisms which admit chaotic attractors.  Consider the family of automorphisms
\begin{equation}\label{HypFamily}
T:K^2\to K^2 \hskip1cm T(x,y)=(ay+b(x^q-x), x)
\end{equation}
where $a,b\in K$ satisfy 
\begin{equation}\label{abAssumptions}
0<|a|<1\text{ and }|b|=q.
\end{equation}
It follows from the assumptions $(\ref{abAssumptions})$ and the congruence $t^q\equiv t\pmod{\pi}$ for all $t\in R$ that $T(R^2)\subseteq R^2$.  Moreover, the restriction $T|_{R^2}:R^2\to R^2$ is nonsurjective; for example $T^{-1}(1,0)=(0,\frac{1}{a})\not\in R^2$ and hence $(1,0)\notin T(R^2)$.  It follows that the (compact) intersection
\begin{equation}\label{AttractorDef}
\Acal_T=\bigcap_{n\geq1}T^n(R^2)
\end{equation}
is an attractor strictly contained in $R^2$ and whose (open) basin of attraction  
\begin{equation}\label{BasinDef}
\Bcal_T=\bigcup_{n\geq1}T^{-n}(R^2)
\end{equation}
strictly contains $R^2$.  

One should expect $T$ to have hyperbolic and chaotic dynamics on $\Acal_T$, as can be seen by considering the characteristic polynomial of the Jacobian matrix of $T$ at a point $(x,y)\in K^2$, which is $p(\lambda)=\lambda^2-b(qx^{q-1}-1)\lambda-a$.  Inspection of the Newton polygon of this polynomial shows that, when $(x,y)\in R^2$, the Jacobian has eigenvalues $\lambda_\min,\lambda_\max\in K$ with $|\lambda_\min|=|a|/|b|<1$ and $|\lambda_\max|=|b|>1$.  

To describe our main results, we first review some basic ideas from symbolic dynamics.  Let $\FF^\ZZ$ be the set of bisequences
$$
(s_k)=(\dots s_{-3}s_{-2}s_{-1}.s_0s_1s_2s_3\dots)
$$
where each $s_k\in\FF$.  This is a compact topological space, and the {\em shift map} on $\FF^\ZZ$ is the homeomorphism $\sigma:\FF^\ZZ\to\FF^\ZZ$ defined by setting the $k$-th term of $\sigma((s_k))$ equal to $s_{k+1}$; that is
\begin{equation}\label{ShiftMap}
\sigma((\dots s_{-3}s_{-2}s_{-1}.s_0s_1s_2s_3\dots))=(\dots s_{-2}s_{-1}s_0.s_1s_2s_3s_4\dots).
\end{equation}

\begin{thm}\label{MainThmIntro}
The automorphism $T:K^2\to K^2$ restricts to a map $T:\Acal_T\to\Acal_T$ which is topologically conjugate to the shift map $\sigma:\FF^\ZZ\to\FF^\ZZ$.  More precisely, there exists a homeomorphism $\omega:\FF^\ZZ\to \Acal_T$ such that $\omega\circ\sigma=T\circ\omega$.
\end{thm}

\begin{equation*}\label{TopConjDiagram}
\begin{CD}
\FF^\ZZ  @> \sigma >>   \FF^\ZZ \\ 
@V \omega VV                                    @VV \omega V \\ 
\Acal_T @> T >> \Acal_T
\end{CD}
\end{equation*}

\bigskip

We prove Theorem~\ref{MainThmIntro} in $\S$~\ref{TopConjSect}, using an argument related to the Smale horseshoe map (\cite{devaney:dynamicsbook} $\S$2.3) going back to the study by Devaney-Nitecki \cite{MR539548} of real H\'enon maps, and which was developed further in the non-Archimedean context by Allen-DeMark-Petsche \cite{MR3757169}.  Woodcock-Smart \cite{MR1678087} considered the automorphism of $R^2$ corresponding to the case $R=\ZZ_p$, $a=1$, $b=1/p$ in $(\ref{HypFamily})$; in this case $\Acal_T=\ZZ_p^2$ is no longer an attractor, but the topological conjugacy to the shift map still holds, and they use this map to create a pseudo-random number generator.  See also the related results of Benedetto-Briend-Perdry \cite{MR2394889} for (non-invertible) quadratic polynomial maps in one variable.

In $\S$~\ref{EquidistSect}, we use Theorem~\ref{MainThmIntro} to obtain a unit Borel measure $\mu_T$ supported on $\Acal_T$ by pushing forward the uniform Bernoulli measure on $\FF^\ZZ$ via the homeomorphism $\omega:\FF^\ZZ\to\Acal_T$.  It is then a simple matter to use the Birkhoff ergodic theorem to show that $\mu_T$-almost all points in $\Acal_T$ have the property that their forward $T$-orbits are $\mu_T$-equidistributed.  Going further, we adapt an argument using the idea of stable manifolds in smooth dynamics to obtain the following result.  Borrowing terminology from dynamical systems on real manifolds, let us say that a $T$-invariant unit Borel measure $\mu$ on $K^2$ is a {\em physical measure} with respect to $T$ if there exists a set $G\subseteq K^2$ of positive Haar measure such that every point in $G$ has $\mu$-equidistributed forward orbit. 

\begin{thm}\label{EquidistThmIntro}
There exists a subset $G\subseteq \Bcal_T$ of full Haar measure in $\Bcal_T$ such that, for all $(x,y)\in G$, the forward orbit $\{T^n(x,y)\}_{n=0}^{\infty}$ is $\mu_T$-equidistributed.  In particular, $\mu_T$ is a physical measure.
\end{thm}

Theorem~\ref{EquidistThmIntro} suggests that $\mu_T$ is a non-Archimedean analogue of the SRB measures (Sinai-Ruelle-Bowen) in smooth dynamics \cite{MR1933431}.  We point out that there are plenty of points in $\Bcal_T$ whose forward orbits are not $\mu_T$-equidistributed; for example, Theorem~\ref{MainThmIntro} implies that $T$-periodic points are dense in $\Acal_T$, and any of the uncountably many points in the stable manifold associated to such a point cannot have a $\mu_T$-equidistributed forward orbit.  See $\S$~\ref{EquidistSect} for more details on these assertions.

Finally, a close analysis of the proof of Theorem~\ref{MainThmIntro}, along with a regularity property of the measure $\mu_T$, allow us give a precise calculation of the Hausdorff dimension of the attractor $\Acal_T$.

\begin{thm}\label{DimThmIntro}
The attractor $\Acal_T$ has Hausdorff dimension 
$$
\dim\Acal_T=1+\frac{1}{1+\log_q(1/|a|)}.
$$  
\end{thm}

In particular, we point out that $\log_q(1/|a|)>0$ because of the assumption $(\ref{abAssumptions})$, and therefore $1<\dim\Acal_T<2$.  

It is interesting to note that one of our preliminary results, Lemma~\ref{HorTubesLem}, implies that $T$ maps $R^2$ bijectively onto a disjoint union of $q$ thin neighborhoods of graphs of functions $g:R\to R$ in $R^2$.  Iterating this lemma, we see that $\Acal_T$ is reminiscent of the well-known Smale-Williams solenoid attractor $\Acal=\cap_{n\geq0}T^n(X)$ associated to a map $T:X\to X$ on a solid torus $X$ in $\RR^3$ which embeds $X$ injectively into itself and wraps around itself $q\geq2$ times (\cite{MR0228014} $\S$~I.9).

We thank Rob Benedetto for helpful suggestions leading to simplifications of the proofs of the results of $\S$~\ref{TopConjSect}.

\section{Notation and preliminaries}

Recall that the absolute value on $K$ is normalized so that $|\pi|=1/q$, and therefore $|K^\times|=q^\ZZ$ is the value group of $K$.  We define the non-Archimedean norm $\|\cdot\|$ on $K^2$ by setting $\|(x_0,y_0)\|=\max(|x_0|,|y_0|)$, and for each $r\in q^\ZZ$, we define
\begin{equation*}
\begin{split}
D_r(x_0) & =\{x\in K\mid |x-x_0|\leq r\} \\
B_r(x_0,y_0) & =\{(x,y)\in K^2\mid \|(x,y)-(x_0,y_0)\|\leq r\},
\end{split}
\end{equation*}
the disc in $K$ centered at $x_0$ of radius $r$, and ball in $K^2$ centered at $(x_0,y_0)$ of radius $r$, respectively.

Let $f:R\to R$ be a function.  We say $f$ is {\em $C$-Lipschitz} for some $C>0$ if $|f(t_1)-f(t_2)|\leq C|t_1-t_2|$ for all $t_1,t_2\in R$.  Denote by 
\begin{equation*}
\begin{split}
H(f) & =\{(t,f(t))\in R^2\mid t\in R\} \\
V(f) & =\{(f(t),t)\in R^2\mid t\in R\} 
\end{split}
\end{equation*}
the {\em horizontal curve} and {\em vertical curve} in $R^2$ associated to $f$, and for each $0<\delta \leq1$, define $\delta$-neighborhoods of $H(f)$ and $V(f)$ by
\begin{equation*}
\begin{split}
H_\delta(f) & =\{(t,f(t)+\theta)\in R^2\mid t\in R, |\theta|\leq \delta\} \\ V_\delta(f) & =\{(f(t)+\theta,t)\in R^2\mid t\in R, |\theta|\leq \delta\}
\end{split}
\end{equation*}

\begin{lem}\label{LipschitzTubes}
Let $f,g:R\to R$ be $(1/q)$-Lipschitz functions and let $0<\delta\leq1$ with $\delta\in q^\ZZ$.  
\begin{itemize}
\item[{\bf (a)}]  If $r\leq\delta$ and $(x_0,y_0)\in H_\delta(f)$, then $B_r(x_0,y_0)\subseteq H_\delta(f)$; similarly for $V_\delta(f)$.
\item[{\bf (b)}]  $H_\delta(f)$ is a union of $1/\delta$ balls of radius $\delta$; similarly for $V_\delta(f)$.
\item[{\bf (c)}]  $V(f)\cap H(g)$ contains exactly one point.
\item[{\bf (d)}]  $V_\delta(f)\cap H_\delta(g)$ is a ball of radius $\delta$.
\end{itemize}
\end{lem}
\begin{proof}
{\bf (a)}  Since $(x_0,y_0)\in H_\delta(f)$, we have $|f(x_0)-y_0|\leq\delta$.  Given $(x_1,y_1)\in B_r(x_0,y_0)$, we have $|x_1-x_0|\leq r$ and $|y_1-y_0|\leq r$, so
\begin{equation*}
\begin{split}
|f(x_1)-y_1| & = |f(x_1)-f(x_0)+f(x_0)-y_0+y_0-y_1|  \\ 	
	& \leq \max(|f(x_1)-f(x_0)|,|f(x_0)-y_0|,|y_0-y_1|) \\
	& \leq \max((1/q)|x_1-x_0|,|f(x_0)-y_0|,|y_0-y_1|) \\
	& =\max(r,\delta)=\delta
\end{split}
\end{equation*}
and thus $(x_1,y_1)\in H_\delta(f)$.  

{\bf (b)}  Writing $\delta=1/q^r$, partition $R$ into $q^r$ discs $D_\delta(x_j)$ with centers $x_1,\dots,x_{q^r}$ and radius $\delta$.  We will show that
\begin{equation*}
H_\delta(f) = \bigcup_{1\leq j\leq q^r}B_\delta(x_j,f(x_j)).
\end{equation*}
That each $B_\delta(x_j,f(x_j))\subseteq H_\delta(f)$ follows from part {\bf (a)}.  Given a point $(x_0,y_0)\in H_\delta(f)$, we have $|y_0-f(x_0)|\leq\delta$.  Since the $D_\delta(x_j)$ partition $R$, we have $x_0\in D_\delta(x_{j_0})$ for some $1\leq j_0\leq q^r$.  Thus $|x_0-x_{j_0}|\leq\delta$, and 
\begin{equation*}
\begin{split}
|y_0-f(x_{j_0})| & = |y_0-f(x_0)+f(x_0)-f(x_{j_0}) |  \\ 	
	& \leq \max(|y_0-f(x_0)|,|f(x_0)-f(x_{j_0})) \\
	& \leq \max(|y_0-f(x_0)|,(1/q)|x_0-x_{j_0}|) \leq\delta 
\end{split}
\end{equation*}
verifying that $(x_0,y_0)\in B_\delta(x_{j_0},f(x_{j_0}))$. 

{\bf (c)}  This follows as in \cite{MR3757169} Lemma 23.  Briefly, the maps $f\circ g:R\to R$ and $g\circ f:R\to R$ are contractions and therefore have unique fixed points $x_0$ and $y_0$ in $R$, respectively.  Thus $y_0=g(x_0)$ and $x_0=f(y_0)$, and $V(f)\cap H(g)=\{(x_0,y_0)\}$.

{\bf (d)}  By part {\bf (c)}, $V(f)\cap H(g)=\{(x_0,y_0)\}$, and thus $f(y_0)=x_0$ and $g(x_0)=y_0$.  We will show that $V_\delta(f)\cap H_\delta(g)=B_\delta(x_0,y_0)$.  That $B_\delta(x_0,y_0)\subseteq V_\delta(f)\cap H_\delta(g)$ follows from part {\bf (a)}.  Converesly, if $(x,y)\in V_\delta(f)\cap H_\delta(g)$, then we have $|f(y)-x|\leq\delta$ and $|g(x)-y|\leq\delta$, and so 
\begin{equation*}
\begin{split}
|x_0-x| & = |f(y_0)-f(y)+f(y)-x| \leq \max((1/q)|y_0-y|,\delta) 
\end{split}
\end{equation*}
and 
\begin{equation*}
\begin{split}
|y_0-y| & = |g(x_0)-g(x)+g(x)-y| \leq \max((1/q)|x_0-x|,\delta) 
\end{split}
\end{equation*}
In particular,
\begin{equation*}
\begin{split}
|x_0-x| & \leq \max((1/q)^2|x_0-x|,(1/q)\delta,\delta)=\delta.
\end{split}
\end{equation*}
The evaluation of this maximum follows from the fact that the alternatives lead to the absurdities $|x_0-x| \leq (1/q)^2|x_0-x|$ and $(1/q)\delta>\delta$.  By a symmetrical argument, we also have $|y_0-y|\leq\delta$ and we have checked that $V_\delta(f)\cap H_\delta(g)\subseteq B_\delta(x_0,y_0)$.
\end{proof}

\section{The topological conjugacy to the shift map}\label{TopConjSect}

In this section we prove Theorem~\ref{MainThmIntro}.  Our approach follows a close parallel with the proof of Theorem 28 of Allen-DeMark-Petsche \cite{MR3757169}, and in particular, Lemmas~\ref{VertTubesLem}, \ref{HorTubesLem}, \ref{ForwardTrajectoryLemma}, and \ref{BackwardTrajectoryLemma} are suitably modified versions of Lemmas 24, 25, 26, and 27 of \cite{MR3757169}.  However, in the current paper the condition $T(R^2)\subseteq R^2$ gives $\Acal_T$ the structure of an attractor and leads to an asymmetry in the forward and backward dynamics of $T$ that is not present in Theorem 28 of \cite{MR3757169}.

Given $b\in K$ satisfying $|b|=q$, it is useful to define the function
\begin{equation}
\phi:K\to K \hskip1cm \phi(t)=b(t^q-t).
\end{equation}
We may then simplify the expressions for $T$ and $T^{-1}$ as 
\begin{equation}
T(x,y)=(ay+\phi(x), x) \hskip1cm \textstyle T^{-1}(x,y)=(y,\frac{1}{a}(x-\phi(y))).
\end{equation}

\begin{lem}\label{MainPhiLemma}
We have $\phi(R)\subseteq R$, and in particular $T(R^2)\subseteq R^2$.  If $t_1,t_2\in R$ satisfy $|t_1-t_2|\leq 1/q$, then $|\phi(t_1)-\phi(t_2)|=q|t_1-t_2|$. 
\end{lem}
\begin{proof}
Since the multiplicative group $\FF^\times$ has order $q-1$, we have the congruence $t^{q-1}\equiv1\pmod{\pi}$ for all $t\not\equiv0\pmod{\pi}$ in $R$, and hence $t^{q}\equiv t\pmod{\pi}$ for all $t\in R$.  It follows that $|t^p-t|\leq1/q$, and hence $|\phi(t)|=|b(t^q-t)|\leq1$ for all $t\in R$, verifying that $\phi(R)\subseteq R$.

Since $\FF=R/\pi R$ has order $q$, we have 
\begin{equation}\label{qIsZeroModpi}
q\equiv0\pmod{\pi}.
\end{equation}  
Given $t_1,t_2\in R$ with $|t_1-t_2|\leq 1/q$ we have $t_1\equiv t_2\pmod{\pi}$ and therefore
\begin{equation*}
t_1^{q-1}+t_1^{q-2}t_2+t_1^{q-3}t_2^2+\dots+t_2^{q-1} \equiv qt_1^{q-1}\equiv 0\pmod{\pi},
\end{equation*}
and we conclude
\begin{equation*}
\begin{split}
|\phi(t_1)-\phi(t_2)| & = |b||t_1^q-t_2^q-(t_1-t_2)| \\
	& = q|t_1^{q-1}+t_1^{q-2}t_2+t_1^{q-3}t_2^2+\dots+t_2^{q-1}-1||t_1-t_2| \\
	& = q|t_1-t_2|.
\end{split}
\end{equation*}
\end{proof}

In a slight abuse of notation, we use $\FF$ to denote both the residue field $R/\pi R$ as well as a complete set of coset representatives in $R$ for this quotient.  Since $|\pi|=1/q$, the congruence classes modulo $\pi$ in $R$ are the same as discs of radius $1/q$ in $R$, and so for each $s\in\FF$ the disc $D_{1/q}(s)=\{x\in R \mid |x-s|\leq1/q \}$ does not depend on the choice of coset representative.

The following lemma states that the inverse image under $T$ of a $\delta$-neighborhood of a vertical curve meets $R^2$ at a union of thinner neighborhoods of vertical curves.  Then, Lemma~\ref{HorTubesLem} gives an analogous result but for the $T$-image of neighborhoods of horizontal curves.  

\begin{lem}\label{VertTubesLem}
Let $f:R\to R$ be a $(1/q)$-Lipschitz function and let $0<\delta \leq1$.  Then for each $s\in\FF$ there exists a $(1/q)$-Lipschitz function $f^s:R\to R$ such that $f^s(t)\in D_{1/q}(s)$ for all $t\in R$ and
\begin{equation}\label{VertTubesLemIdentity}
T^{-1}(V_\delta (f))\cap R^2=\bigcup_{s\in\FF}V_{(1/q)\delta}(f^s).
\end{equation}
\end{lem}
\begin{proof}
Fix $t\in R$ and $s\in\FF$, and define a $(1/q)$-Lipschitz function
\begin{equation*}
\begin{split}
& F_{t}^s:D_{1/q}(s)\to D_{1/q}(s) \\
& F_{t}^s(x) = x+\frac{at+\phi(x)-f(x)}{b}.
\end{split}
\end{equation*}
If $x\in D_{1/q}(s)$, then $|F_{t}^s(x)-s|=|x-s+\frac{at+\phi(x)-f(x)}{b}|\leq 1/q$, verifying that $F_{t}^s(D_{1/q}(s))\subseteq D_{1/q}(s)$.  To check the Lipschitz condition, note that for distinct $x_1,x_2\in D_{1/q}(s)$, using $(\ref{qIsZeroModpi})$ we have
\begin{equation*}
\begin{split}
\frac{1}{b}\frac{\phi(x_1)-\phi(x_2)}{x_1-x_2} & =\frac{(x_1^q-x_1)-(x_2^q-x_2)}{x_1-x_2}  \\
	& =x_1^{q-1}+x_1^{q-2}x_2+\dots+x_2^{q-1}-1 \\
	& \equiv qs^{q-1}-1 \equiv-1\pmod{\pi R}
\end{split}
\end{equation*}
and therefore since $|b|=q$ we have
\begin{equation}\label{PhiDerivative}
\begin{split}
\bigg|b+\frac{\phi(x_1)-\phi(x_2)}{x_1-x_2}\bigg| & \leq1.
\end{split}
\end{equation}
We conclude, using the Lipschitz assumption on $f$, that
$$
\bigg|\frac{F_{t}^s(x_1)-F_{t}^s(x_2)}{x_1-x_2}\bigg|=\frac{1}{|b|}\bigg|b+\frac{\phi(x_1)-\phi(x_2)}{x_1-x_2}-\frac{f(x_1)-f(x_2)}{x_1-x_2}\bigg|\leq\frac{1}{|b|}=\frac{1}{q},
$$
and thus $F_{t}^s$ is $(1/q)$-Lipschitz.

Since $F_{t}^s$ is contracting, it has a unique fixed point in $D_{1/q}(s)$; call this point $f^s(t)$.  We have constructed the function $f^s:R\to D_{1/q}(s)$, and this function satisfies 
\begin{equation}\label{FixedPointIdentity}
at+\phi(f^s(t))-f(f^s(t))=0.
\end{equation}
It is straightforward to check that $f^s$ is Lipschitz using Lemma~\ref{MainPhiLemma} and the Lipschitz assumption on $f$.

Elementary calculations using $(\ref{FixedPointIdentity})$ show that for each $s\in \FF$, we have
\begin{equation}\label{TubesIdentity}
\{(x,y)\in R^2\mid x\in D_{1/q}(s),|ay+\phi(x)-f(x)|\leq\delta \} =V_{(1/q)\delta}(f^s),
\end{equation}
and it follows from $(\ref{TubesIdentity})$ that
\begin{equation*}
\begin{split}
T^{-1}(V_\delta (f))\cap R^2 & = \{(x,y)\in R^2\mid T(x,y)\in V_\delta (f)\} \\
	& = \{(x,y)\in R^2\mid |ay+\phi(x)-f(x)|\leq\delta \} \\
	& = \bigcup_{s\in\FF}\{(x,y)\in R^2\mid x\in D_{1/q}(s),|ay+\phi(x)-f(x)|\leq\delta \} \\
	& = \bigcup_{s\in\FF}V_{(1/q)\delta}(f^s).
\end{split}
\end{equation*}
\end{proof}

\begin{lem}\label{HorTubesLem}
Let $g:R\to R$ be a $(1/q)$-Lipschitz function and let $0<\epsilon\leq1$.  Then for each $s\in\FF$ there exists a $(1/q)$-Lipschitz function $g^s:R \to R$ such that $g^s(t)\in D_{1/q}(s)$ for all $t\in  R$ and
\begin{equation*}
T(H_\epsilon(g))=\bigcup_{s\in\FF}H_{(|a|/q)\epsilon}(g^s).
\end{equation*}
\end{lem}
\begin{proof}
Fix $t\in R$ and $s\in\FF$, and define a $(1/q)$-Lipschitz function
\begin{equation*}
\begin{split}
& G_{t}^s:D_{1/q}(s)\to D_{1/q}(s) \\
& G_{t}^s(y) = y-\frac{t-\phi(y)-ag(y)}{b}.
\end{split}
\end{equation*}

%If $y\in D_{1/q}(s)$, then $|G_{t}^s(y)-s|=|y-s-\frac{t-\phi(y)-ag(y)}{b}|\leq 1/q$, verifying that $G_{t}^s(D_{1/q}(s))\subseteq D_{1/q}(s)$.  To check the Lipschitz condition, we again appeal to $(\ref{PhiDerivative})$ and the Lipschitz assumption on $g$, and we note that for distinct $y_1,y_2\in D_{1/q}(s)$ we have
%$$
%\bigg|\frac{G_{t}^s(y_1)-G_{t}^s(y_2)}{y_1-y_2}\bigg|=\frac{1}{|b|}\bigg|b+\frac{\phi(y_1)-\phi(y_2)}{y_1-y_2}+a\frac{g(y_1)-g(y_2)}{y_1-y_2}\bigg|\leq\frac{1}{|b|}=\frac{1}{q}.
%$$

That $G_{t}^s(D_{1/q}(s))\subseteq D_{1/q}(s)$ and that $G_{t}^s$ is $(1/q)$-Lipschitz follows from arguments similar to those in the proof of Lemma~\ref{VertTubesLem}.  Since $G_{t}^s$ is contracting, it has a unique fixed point in $D_{1/q}(s)$; call this point $g^s(t)$.  We have constructed the function $g^s:R\to D_{1/q}(s)$, and it satisfies 
\begin{equation}\label{HorFixedPointIdentity}
t-\phi(g^s(t))-ag(g^s(t))=0.
\end{equation}
It is straightforward to check that $g^s$ is Lipschitz using Lemma~\ref{MainPhiLemma} and the Lipschitz assumption on $g$.

Elementary calculations using $(\ref{HorFixedPointIdentity})$ show that for each $s\in \FF$, we have
\begin{equation}\label{HorTubesIdentity}
\{(x,y)\in R^2\mid y\in D_{1/q}(s),|x-\phi(y)-ag(y)|\leq|a|\epsilon\} =H_{(|a|/q)\epsilon}(g^s).
\end{equation}
Using $(\ref{HorTubesIdentity})$ and the fact that $T(H_\epsilon(g))\subseteq T(R^2)\subseteq R^2$ we conclude the desired identity
\begin{equation*}
\begin{split}
T(H_\epsilon(g)) & = \{(x,y)\in R^2\mid T^{-1}(x,y)\in H_\epsilon(g)\} \\
	& = \{(x,y)\in R^2\mid |x-\phi(y)-ag(y)|\leq|a|\epsilon\} \\
	& = \bigcup_{s\in\FF}\{(x,y)\in R^2\mid y\in D_{1/q}(s),|x-\phi(y)-ag(y)|\leq|a|\epsilon\} \\
	& = \bigcup_{s\in\FF}H_{(|a|/q)\epsilon}(g^s).
\end{split}
\end{equation*}

\end{proof}

Observe that the collection of $q$ sets
\begin{equation}\label{VerticalPartition}
\{D_{1/q}(s)\times R\mid s\in \FF\}
\end{equation}
forms a partition of $R^2$.  Since $T(R^2)\subseteq R^2$, the forward $T$-orbit of each point of $R^2$ follows a trajectory through the $q$ sets in the partition $(\ref{VerticalPartition})$.  The following lemma shows that all possible trajectories occur, and the set of points with a given trajectory is a vertical curve in $R^2$.  Then, Lemma~\ref{BackwardTrajectoryLemma} gives an analogous statement for backward orbits of points in $\Acal_T$ in terms of horizontal curves.

\begin{lem}\label{ForwardTrajectoryLemma}
There exists a family of $(1/q)$-Lipschitz functions $f^{(s_0s_1s_2\dots)}:R\to R$, indexed by the set of all sequences $(s_0s_1s_2\dots)$, where each $s_k\in\FF$, such that
\begin{equation}\label{CompleteVerticalTrajectorySetsAreCurves}
V(f^{(s_0s_1s_2\dots)})=\{(x,y)\in R^2\mid T^k(x,y)\in D_{1/q}(s_k)\times R\text{ for all }k\geq0\}.
\end{equation}
Moreover, 
\begin{equation}\label{R2Partition}
R^2=\bigcup_{s_0,s_1,s_2,\dots\in\FF}V(f^{(s_0s_1s_2\dots)}).
\end{equation}
\end{lem}

\begin{proof}
We first construct a family of $(1/q)$-Lipschitz functions $f^{(s_0s_1s_2\dots)}_{n}:R\to R$, indexed by sequences $(s_0s_1s_2\dots)$ in $\FF$ and integers $n\geq0$.  When $n=0$, we define $f^{(s_0s_1s_2\dots)}_0\equiv s_0$, and thus in this case
\begin{equation}\label{FirstVerticalTubes}
V_{1/q}(f_0^{(s_0s_1s_2\dots)})=D_{1/q}(s_0)\times R.
\end{equation} 

To ease notation, set $\delta_n=1/q^{n+1}$.  Fix $n\geq0$ and assume that the functions $f^{(s_0s_1s_2\dots)}_{n}:R\to R$ have been constructed for all sequences $(s_0s_1s_2\dots)$ in $\FF$.  Given a sequence $(s_0s_1s_2\dots)$, we apply Lemma~\ref{VertTubesLem} with $f=f_n^{(s_1s_2s_3\dots)}$ and $\delta =\delta_n$.  We obtain $(1/q)$-Lipschitz functions $f^s:R\to D_{1/q}(s)$, and we define $f^{(s_0s_1s_2\dots)}_{n+1}=f^{s_0}$.  We then have, for each fixed choice of $s_1,s_2,\dots$ in $\FF$, the identity
\begin{equation}\label{GeneralTwoVerticalTubes}
T^{-1}(V_{\delta_n}(f_n^{(s_1s_2s_3\dots)}))\cap R^2=\bigcup_{s_0\in\FF}V_{\delta_{n+1}}(f^{(s_0s_1s_2\dots)}_{n+1}).
\end{equation}

An induction argument using $(\ref{FirstVerticalTubes})$ and $(\ref{GeneralTwoVerticalTubes})$ shows that for each sequence $(s_0s_1s_2\dots)$, we have
\begin{equation}\label{TrajectorySetsAreTubes}
V_{\delta_{n}}(f^{(s_0s_1s_2\dots)}_{n})=\{(x,y)\in R^2\mid T^k(x,y)\in D_{1/q}(s_k)\times R\text{ for all }0\leq k\leq n\}.
\end{equation}
In other words, $(\ref{TrajectorySetsAreTubes})$ says that $V_{\delta_{n}}(f^{(s_0s_1s_2\dots)}_{n})$ is the set of points in $R^2$ whose partial $T$-orbit $\{T^k(x,y)\}_{k=0}^{n}$ follows a particular trajectory through the $q$ disjoint sets $D_{1/q}(s)\times R$ for $s\in\FF$.  

From $(\ref{TrajectorySetsAreTubes})$ it is clear that
$$
V_{\delta_{n+1}}(f^{(s_0s_1s_2\dots)}_{n+1})\subseteq V_{\delta_{n}}(f^{(s_0s_1s_2\dots)}_{n}),
$$
from which it follows that the limit $f^{(s_0s_1s_2\dots)}(t) := \lim_{n\to+\infty}f^{(s_0s_1s_2\dots)}_{n}(t)$ exists, and a standard argument shows that a limit of $(1/q)$-Lipschitz functions is $(1/q)$-Lipschitz.  We conclude using $(\ref{TrajectorySetsAreTubes})$ that 
$V(f^{(s_0s_1s_2\dots)}) =\bigcap_{n\geq0}V_{\delta_{n}}(f^{(s_0s_1s_2\dots)}_{n})$, obtaining $(\ref{CompleteVerticalTrajectorySetsAreCurves})$.

Since $T(R^2)\subseteq R^2$, it follows that every point in $R^2$ has some forward trajectory through the $q$ sets in the partition $\{D_{1/q}(s)\times R\mid s\in\FF\}$ of $R^2$, and $(\ref{R2Partition})$ follows immediately. 
\end{proof}

\begin{lem}\label{BackwardTrajectoryLemma}
There exists a family of $(1/q)$-Lipschitz functions $g^{(\dots s_{-3}s_{-2}s_{-1})}:R\to R$, indexed by the set of all sequences $(\dots s_{-3}s_{-2}s_{-1})$, where each $s_k\in\FF$, such that
\begin{equation}\label{CompleteHorTrajectorySetsAreCurves}
H(g^{(\dots s_{-3}s_{-2}s_{-1})})=\{(x,y)\in R^2\mid T^{k+1}(x,y)\in R\times D_{1/q}(s_k)\text{ for all }k\leq -1\}.
\end{equation}
Moreover, 
\begin{equation}\label{ATPartition}
\Acal_T=\bigcup_{\dots, s_{-3},s_{-2},s_{-1}\in\FF}H(g^{(\dots s_{-3}s_{-2}s_{-1})}).
\end{equation}
\end{lem}
\begin{proof}
We first construct a family of $(1/q)$-Lipschitz functions $g^{(\dots s_{-3}s_{-2}s_{-1})}_{n}:R\to R$, indexed by the set of sequences $(\dots s_{-3}s_{-2}s_{-1})$ in $\FF$ and integers $n\geq0$.  When $n=0$, we define $g^{(\dots s_{-3}s_{-2}s_{-1})}_0\equiv s_{-1}$, and thus in this case
\begin{equation}\label{FirstHorTubes}
H_{1/q}(g_0^{(\dots s_{-3}s_{-2}s_{-1})})=R \times D_{1/q}(s_{-1}).
\end{equation} 

To ease notation, set $\epsilon_n=|a|^n/q^{n+1}$.  Fix $n\geq0$ and assume that the functions $g^{(\dots s_{-3}s_{-2}s_{-1})}_{n}:R\to R$ have been constructed for all sequences $(\dots s_{-3}s_{-2}s_{-1})$ in $\FF$.  Given a sequence $(\dots s_{-3}s_{-2}s_{-1})$, we apply Lemma~\ref{HorTubesLem} with $g=g_n^{(\dots s_{-4}s_{-3}s_{-2})}$ and $\epsilon=\epsilon_n$, and we obtain $(1/q)$-Lipschitz functions $g^s:R\to D_{1/q}(s)$.  Setting $g^{(\dots s_{-3}s_{-2}s_{-1})}_{n+1}=g^{s_{-1}}$, for each fixed choice of $\dots,s_{-4},s_{-3},s_{-2}$ in $\FF$ we have
\begin{equation}\label{GeneralTwoHorTubes}
T(H_{\epsilon_n}(g_n^{(\dots s_{-4}s_{-3}s_{-2})}))=\bigcup_{s_{-1}\in\FF}H_{\epsilon_{n+1}}(g^{(\dots s_{-3}s_{-2}s_{-1})}_{n+1}).
\end{equation}

An induction argument using $(\ref{FirstHorTubes})$ and $(\ref{GeneralTwoHorTubes})$ shows that for a sequence $(\dots s_{-3}s_{-2}s_{-1})$, we have
\begin{equation}\label{HorTrajectorySetsAreTubes}
H_{\epsilon_{n}}(g^{(\dots s_{-3}s_{-2}s_{-1})}_{n})=\{(x,y)\in R^2\mid T^{k+1}(x,y)\in R\times D_{1/q}(s_k)\text{ for all }-n-1\leq k\leq -1\},
\end{equation}
from which it follows that $H_{\epsilon_{n+1}}(g^{(\dots s_{-3}s_{-2}s_{-1})}_{n+1})\subseteq H_{\epsilon_{n}}(g^{(\dots s_{-3}s_{-2}s_{-1})}_{n})$.  The limit 
$$
g^{(\dots s_{-3}s_{-2}s_{-1})}(t) = \lim_{n\to+\infty}g^{(\dots s_{-3}s_{-2}s_{-1})}_{n}(t)
$$
exists and is $(1/q)$-Lipschitz, and since $H(g^{(\dots s_{-3}s_{-2}s_{-1})})=\bigcap_{n\geq0}H_{\epsilon_{n}}(g^{(\dots s_{-3}s_{-2}s_{-1})}_{n})$ we obtain $(\ref{CompleteHorTrajectorySetsAreCurves})$.  

Since $\Acal_T=\cap_{n\geq0}T^n(R^2)$ is the set of all points in $R^2$ whose entire (backward) orbit is contained in $R^2$, and since each such point must have some backward trajectory through the $q$ sets in the partition $\{R\times D_{1/q}(s)\mid s\in\FF\}$ of $R^2$, we obtain $(\ref{ATPartition})$.
\end{proof}

\begin{proof}[Proof of Theorem~\ref{MainThmIntro}]
Fix a bisequence $\s=(\dots s_{-2}s_{-1}.s_0s_1s_2\dots)\in\FF^\ZZ$.  Then the functions $f^{(s_0s_1s_2\dots)}:R\to R$ and $g^{(\dots s_{-3}s_{-2}s_{-1})}:R\to R$ constructed in Lemmas~\ref{ForwardTrajectoryLemma} and \ref{BackwardTrajectoryLemma} are $(1/q)$-Lipschitz.  It follows from Lemma~\ref{LipschitzTubes} {\bf (c)} that the curves $H(g^{(\dots s_{-3}s_{-2}s_{-1})})$ and $V(f^{(s_0s_1s_2\dots)})$ intersect at a single point in $R^2$.  Denoting this point of intersection by $\omega(\s)$, we obtain a function $\omega:\FF^\ZZ\to \Acal_T$.

Given a point $(x,y)\in \Acal_T$, it follows from the fact that $T(R^2)\subseteq R^2$ and $(\ref{AttractorDef})$ that every point in its $T$-orbit is contained in $R^2$; hence every point in its orbit is contained in one of the $q$ sets $R\times D_{1/q}(s)$ for $s\in\FF$, and similarly every point in its orbit is contained in one of the $q$ sets $D_{1/q}(s)\times R$ for $s\in\FF$.  Define $\s=(s_k)\in\FF^\ZZ$ by $T^k(x,y)\in D_{1/q}(s_k)\times R$ for $k\geq0$, and $T^{k+1}(x,y)\in R\times D_{1/q}(s_k)$ for $k\leq-1$.  The function $(x,y)\mapsto \s$ defines an inverse $\omega^{-1}:\Acal_T\to\FF^\ZZ$ by Lemma~\ref{ForwardTrajectoryLemma} and Lemma~\ref{BackwardTrajectoryLemma}, and so $\omega$ is bijective.  

That $\omega$ and $\omega^{-1}$ are continuous follows as in \cite{MR3757169} Theorem 28.  To summarize this argument, recall that the cylinder sets
$$
\Sigma_{t_{-N},\dots,t_N}=\{\s=(s_k)\in \FF^\ZZ\mid s_k=t_k\text{ for all }|k|\leq N\}.
$$
form a neighborhood base for the topology on $\FF^\ZZ$.  It follows from $(\ref{TrajectorySetsAreTubes})$ and $(\ref{HorTrajectorySetsAreTubes})$ that $\omega(\Sigma_{t_{-N},\dots,t_N})$ is $V_\delta(f)\cap H_\epsilon(g)$ for certain functions $f,g:R\to R$ and certain radii $\delta,\epsilon>0$.  As $V_\delta(f)$ and $H_\epsilon(g)$ are topologically open, this shows that $\omega^{-1}$ is continuous.  It is a standard exercise in basic topology to show that a continuous bijection of compact sets has a continuous inverse, and thus $\omega$ is a homeomorphism.

To see that $\omega\circ\sigma=T\circ\omega$, fix $\s=(s_k)\in\FF^\ZZ$ and let $\t=(t_k)=\sigma(\s)$; thus $t_k=s_{k+1}$.  By Lemma~\ref{ForwardTrajectoryLemma} and Lemma~\ref{BackwardTrajectoryLemma}, we have 
\begin{equation*}
\begin{split}
T^k(\omega(\s)) & \in  D_{1/q}(s_k)\times R \hskip1cm \text{for all $k\geq0$} \\
T^{k+1}(\omega(\s)) & \in  R\times D_{1/q}(s_k) \hskip1cm \text{for all $k\leq-1$}.
\end{split}
\end{equation*}
It follows immediately that
\begin{equation*}
\begin{split}
T^k(T(\omega(\s))) & \in  D_{1/q}(t_k)\times R \hskip1cm \text{for all $k\geq0$} \\
T^{k+1}(T(\omega(\s))) & \in  R\times D_{1/q}(t_k) \hskip1cm \text{for all $k\leq-2$}.
\end{split}
\end{equation*}
It is clear from $(\ref{HypFamily})$ and the fact that $T(R^2)\subseteq R^2$ that $T(D_{1/q}(s)\times R)\subseteq R\times D_{1/q}(s)$ for all $s\in\FF$, and therefore 
$$
T(\omega(\s))\in T(D_{1/q}(s_0)\times R) \subseteq R\times D_{1/q}(s_0)=R\times D_{1/q}(t_{-1}). 
$$
We conclude that
$$
T(\omega(\s))\in H(g^{(\dots t_{-3}t_{-2}t_{-1})})\cap V(f^{(t_0t_1t_2\dots)})
$$ 
and therefore using Lemma~\ref{ForwardTrajectoryLemma} and Lemma~\ref{BackwardTrajectoryLemma} we find that $T(\omega(\s))=\omega(\t)$; in other words, $T(\omega(\s))=\omega(\sigma(\s))$.
\end{proof}

\section{The Bernoulli measure and equidistribution of forward orbits}\label{EquidistSect}

\subsection{Symbolic dynamics and the Bernoulli measure on $\FF^\ZZ$}

Recall that $\FF^\ZZ$ denotes the set of bisequences
$$
(s_k)=(\dots s_{-3}s_{-2}s_{-1}.s_0s_1s_2s_3\dots)
$$
where each $s_k\in\FF$.  The {\em cylinder sets}
\begin{equation}\label{CylinderSet}
\{(s_k)\in\FF^\ZZ\mid s_k=t_k \text{ for all } |k|\leq M\} \hskip1cm (\text{$M\geq0$, $t_{-M},\dots,t_M\in\FF$})
\end{equation}
are a base of open sets for a compact topology on $\FF^\ZZ$, and the shift map $\sigma:\FF^\ZZ\to\FF^\ZZ$ defined in $(\ref{ShiftMap})$ is a homeomorphism.  

Let $\mu_\sigma$ be the $\sigma$-invariant uniform Bernoulli measure on the Borel $\sigma$-algebra $\Mcal(\FF^\ZZ)$ of $\FF^\ZZ$.  Thus $\FF^\ZZ$ has $\mu_\sigma$-measure $1$, and $\mu_\sigma$ assigns measure $1/q^{2M+1}$ to each cylinder set described in $(\ref{CylinderSet})$.  More generally, given any Borel subset $E\in \Mcal(\FF^\ZZ)$ we have $\mu_\sigma(E)=\inf\sum_j\mu_\sigma(U_j)$, the infimum over all countable covers of $E$ by cylinder sets $U_j$ of the form $(\ref{CylinderSet})$.

Now let $\Mcal_{\geq0}(\FF^\ZZ)$ be the sub-$\sigma$-algebra of $\Mcal(\FF^\ZZ)$ consisting of those Borel subsets $E\subseteq\FF^\ZZ$ with the property that
$$
(\dots s_{-3}s_{-2}s_{-1}.s_0s_1s_2s_3\dots)\in E \hskip5mm \Rightarrow \hskip5mm (\dots s_{-3}'s_{-2}'s_{-1}'.s_0s_1s_2s_3\dots)\in E
$$
for any choice of $\dots ,s'_{-3},s_{-2}',s_{-1}'\in\FF$.  In other words, $\Mcal_{\geq0}(\FF^\ZZ)$ consists of those Borel subsets of $\FF^\ZZ$ which are closed under changing any of the negatively indexed terms of its elements. 

\begin{lem}\label{SubSigmaAlgebra}
Let $E\in \Mcal_{\geq0}(\FF^\ZZ)$ have $\mu_\sigma$-measure zero.  Then for any $\epsilon>0$, $E$ is covered by a countable collection of cylinder sets of the form
\begin{equation}\label{CylinderSet2}
\{(s_k)\in\FF^\ZZ\mid s_k=t_k \text{ for all } 0\leq k\leq M\} \hskip1cm (\text{$M\geq0$, $t_0,\dots,t_M\in\FF$})
\end{equation}
in $\Mcal_{\geq0}(\FF^\ZZ)$, the sum of whose $\mu_\sigma$-measures is $\leq\epsilon$.

\end{lem}
\begin{proof}
Letting $\FF^{\ZZ_{\geq0}}$ denote the space of sequences $(s_0s_1s_2\dots)$ with each $s_k\in\FF$, define a forgetful map $P:\FF^\ZZ\to \FF^{\ZZ_{\geq0}}$ by $P(\dots s_{-3}s_{-2}s_{-1}.s_0s_1s_2s_3\dots)=(s_0s_1s_2s_3\dots)$.  The map $E\mapsto P(E)$ is a measure-preserving bijection from $\Mcal_{\geq0}(\FF^\ZZ)$ to the Borel $\sigma$-algebra $\Mcal(\FF^{\ZZ_{\geq0}})$ of $\FF^{\ZZ_{\geq0}}$.  Thus the statement of the lemma follows at once from the definition of the uniform Bernoulli measure defined on the Borel $\sigma$-algebra of $\FF^{\ZZ_{\geq0}}$.
\end{proof}

\subsection{The equidistriubtion of forward orbits}

Define $\mu_T$ to be unit Borel measure on $K^2$ obtained as the pushforward of the Bernoulli measure $\mu_\sigma$ via the homeomorphism $\omega:\FF^\ZZ\to \Acal_T$ described in Theorem~\ref{MainThmIntro}.  In other words, $\mu_T(A)=\mu_\sigma(\omega^{-1}(A))$ for all Borel subsets $A$ of $K^2$.  The measure $\mu_T$ is supported on $\Acal_T$, and it is $T$-invariant by Theorem~\ref{MainThmIntro} and the fact that $\mu_\sigma$ is $\sigma$-invariant.

We say a point $(x,y)\in K^2$ is {\em $\mu_T$-generic} if its forward orbit $\{T^n(x,y)\}_{n=0}^{+\infty}$ is $\mu_T$-equidistributed, in the sense that for all continuous functions $F:K^2\to\RR$, we have
\begin{equation}\label{GenericDef}
\lim_{N\to+\infty}\frac{1}{N}\sum_{n=0}^{N-1}F(T^n(x,y))=\int F d\mu_T.
\end{equation}

\begin{lem}\label{ErgodicTheoremStepLem}
$\mu_T$-almost all points in $\Acal_T$ are $\mu_T$-generic.
\end{lem}
\begin{proof}
Let $\{B_j\}$ be the (countable) collection of all balls in $R^2$.  Let $A$ be the space of all finite real linear combinations of the characteristic functions $\chi_{B_j}$ of these balls.  As $\chi_{B_j}\chi_{B_{j'}}=\chi_{B_j\cap B_{j'}}$ and $B_j\cap B_{j'}$ is either the empty set or itself a ball, the space $A$ is a point-separating subalgebra of the space $\Ccal(R^2)$ of all continuous real-valued functions on $R^2$, and therefore $A$ is dense in $\Ccal(R^2)$ by the Stone-Weierstrass theorem.

By the topological conjugacy $\omega:\FF^\ZZ\to\Acal_T$ and the fact that the shift map on $\FF^\ZZ$ is ergodic (\cite{MR648108} Thm 1.12) with respect to $\mu_\sigma$, we know that the map $T:\Acal_T\to\Acal_T$ is ergodic with respect to $\mu_T$.  It follows from the Birkhoff ergodic theorem (\cite{MR648108} Thm 1.14) that for each ball $B_j$ in $R^2$, the limit $(\ref{GenericDef})$ holds for $F=\chi_{B_j}$ and for all $(x,y)$ in a set $S_j\subseteq\Acal_T$ of full $\mu_T$-measure in $\Acal_T$.  Thus $S=\cap_jS_j$ has full $\mu_T$-measure in $\Acal_T$ and $(\ref{GenericDef})$ holds for all $F=\chi_{B_j}$ and for all $(x,y)\in S$.  That $(\ref{GenericDef})$ holds for all continuous $F:K^2\to\RR$ and all points $(x,y)\in S$ follows from a standard argument approximating $F$ uniformly in $R^2$ by a real linear combination of the $\chi_{B_j}$, which we omit.
\end{proof}

In order to deduce Theorem~\ref{EquidistThmIntro} from Lemma~\ref{ErgodicTheoremStepLem}, we introduce a certain collection of subsets of $R^2$ which play a role analogous to the stable manifolds in smooth dynamics.  Recall from the proof of Theorem~\ref{MainThmIntro} that, because $T(R^2)\subseteq R^2$, the forward orbit of each point in $R^2$ traverses some trajectory through the $q$ sets in the partition $\{D_{1/q}(s)\times R\mid s\in\FF\}$ of $R^2$.  Thus each point $(x,y)\in R^2$ determines a sequence $(s_0s_1s_2\dots)$ with $s_k\in\FF$ by 
\begin{equation}\label{PointAndTrajectory}
T^k(x,y)\in D_{1/q}(s_k)\times R\text{ for all }k\geq0.
\end{equation}
Define an equivalence relation on $R^2$ in which two points $(x,y),(x',y')\in R^2$ are equivalent if and only if (in an obvious notation) $s_k=s'_k$ for all $k\geq0$.  Denote by $W_{x,y}$ the equivalence class of the point $(x,y)\in R^2$.  From Lemma~\ref{ForwardTrajectoryLemma} we recall that if the point $(x,y)\in R^2$ and the sequence $(s_0s_1s_2\dots)$ are related by $(\ref{PointAndTrajectory})$, we have
$$
W_{x,y}=V(f^{(s_0s_1s_2\dots)}).
$$
In other words, $W_{x,y}$ may also be described as the vertical curve associated to the $(1/q)$-Lipschitz function $f^{(s_0s_1s_2\dots)}:R\to R$ constructed in Lemma~\ref{ForwardTrajectoryLemma}.  

The following lemma states that each equivalence class $W_{x,y}$ consists either entirely of $\mu_T$-generic points or entirely of non-$\mu_T$-generic points. The proof is based on the principle that any two points in $W_{x,y}$ have asymptotically the same forward orbit.

\begin{lem}\label{StabManGeneric}
If $(x,y)\in R^2$ is $\mu_T$-generic, then every point in $W_{x,y}$ is $\mu_T$-generic.
\end{lem}
\begin{proof}
Assume that $(x,y)\in R^2$ is $\mu_T$-generic and let $(x',y')\in W_{x,y}$.  We will show that
\begin{equation}\label{SameAsymptoticOrbit}
\|T^n(x,y)-T^n(x',y')\|\to0 \text{ as }n\to+\infty.
\end{equation}
By hypothesis, there exists a sequence $(s_0s_1s_2\dots)$ in $\FF$ such that both $T^k(x,y)$ and $T^k(x',y')$ are in $D_{1/q}(s_k)\times R$ for all $k\geq0$.  Fixing $n\geq1$, we deduce that both $T^k(T^n(x,y))$ and $T^k(T^n(x',y'))$ are in $D_{1/q}(t_k)\times R$ for all $k\geq-n$, where $t_k=s_{k+n}$.  Thus whenever $k\geq-n$, we have
$$
T^{k+1}(T^n(x,y))\in T(D_{1/q}(t_k)\times R)\subseteq R\times D_{1/q}(t_k)
$$
and likewise for $T^{k+1}(T^n(x',y'))$.  We deduce using $(\ref{TrajectorySetsAreTubes})$ and $(\ref{HorTrajectorySetsAreTubes})$ that both $T^n(x,y)$ and $T^n(x',y')$ are elements of 
$$
V(f^{(t_0t_1t_2\dots)})\cap H_{\epsilon_{n-1}}(g^{(\dots t_{-3}t_{-2}t_{-1})}_{n-1}),
$$
which by Lemma~\ref{LipschitzTubes} {\bf (d)} is contained in a ball of radius $\epsilon_{n-1}$.  As $\epsilon_{n-1}\to0$, we conclude $(\ref{SameAsymptoticOrbit})$.

Given a continuous function $F:K^2\to\RR$, it must be uniformly continuous on $R^2$ by compactness, and so a standard argument shows that $(\ref{SameAsymptoticOrbit})$ implies
\begin{equation}\label{SameAsymptoticOrbit2}
|F(T^n(x,y))-F(T^n(x',y'))|\to0 \text{ as }n\to+\infty.
\end{equation}
We then have
\begin{equation*}
\begin{split}
\bigg|\frac{1}{N}\sum_{n=0}^{N-1}F(T^n(x',y'))-\int F d\mu_T\bigg|  \leq & \frac{1}{N}\sum_{n=0}^{N-1}|F(T^n(x',y'))-F(T^n(x,y))| \\
	& + \bigg|\frac{1}{N}\sum_{n=0}^{N-1}F(T^n(x,y))-\int F d\mu_T\bigg|
\end{split}
\end{equation*}
and therefore
\begin{equation*}
\bigg|\frac{1}{N}\sum_{n=0}^{N-1}F(T^n(x',y'))-\int F d\mu_T\bigg| \to0 \text{ as } N\to+\infty,
\end{equation*}
completing the proof that $(x',y')$ is $\mu_T$-generic.
\end{proof}

\begin{proof}[Proof of Theorem~\ref{EquidistThmIntro}]
Let $\lambda$ denote Haar measure on the Borel $\sigma$-algebra of $K^2$, normalized so that $\lambda(R^2)=1$.  We first show that the set
$$
E=\{(x,y)\in R^2\mid (x,y)\text{ is not $\mu_T$-generic}\}
$$ 
has Haar measure zero.  By Lemma~\ref{ErgodicTheoremStepLem}, $E\cap\Acal_T$ has $\mu_T$-measure zero, and by Lemma~\ref{StabManGeneric}, $E\cap\Acal_T$ is the image under $\omega:\FF^\ZZ\to\Acal_T$ of a set in the sub-$\sigma$-algebra $\Mcal_{\geq0}(\FF^\ZZ)$ of the Borel $\sigma$-algebra $\Mcal(\FF^\ZZ)$.  Fixing $\epsilon>0$, it follows from Lemma~\ref{SubSigmaAlgebra} that there exists a countable collection $\{U_j\}$ of subsets of $\Acal_T$ satisfying the following properties: 
\begin{itemize}
\item $E\cap \Acal_T\subseteq \cup_jU_j$
\item each $U_j$ is the image under $\omega:\FF^\ZZ\to\Acal_T$ of a cylinder set of the form $(\ref{CylinderSet2})$
\item $\sum_j\mu_T(U_j)\leq\epsilon$.
\end{itemize}

Note that for every $(x,y)\in R^2$, the set $W_{x,y}$ contains (uncountably many) points of $\Acal_T$.  Indeed, if the sequence $(s_0s_1s_2\dots)$ is defined by $(\ref{PointAndTrajectory})$, then we may select $\dots, s_{-3},s_{-2},s_{-1}$ in $\FF$ arbitrarily and take $(x',y')=\omega((\dots s_{-3}s_{-2}s_{-1}.s_0s_1s_2\dots))\in\Acal_T\cap W_{x,y}$.  Using Lemma~\ref{StabManGeneric} along with the facts that each $W_{x,y}$ meets $\Acal_T$ and that the $U_j$ cover $E\cap\Acal_T$, we have
\begin{equation}\label{PartitionOfE}
\begin{split}
E & = \bigcup_{(x,y)\in E}W_{x,y} \\
	& = \bigcup_{(x,y)\in E\cap\Acal_T}W_{x,y} \\
	& \subseteq \bigcup_j\bigcup_{(x,y)\in U_j}W_{x,y} \\
\end{split}
\end{equation}
Now, for each $j$, the set $U_j$ is the image under $\omega$ of a cylinder set of the form $(\ref{CylinderSet2})$, and using $(\ref{TrajectorySetsAreTubes})$ we therefore have
\begin{equation*}
\begin{split}
\bigcup_{(x,y)\in U_j}W_{x,y} & =\{(x,y)\in R^2\mid T^k(x,y)\in D_{1/q}(t_k)\times R\text{ for all }0\leq k\leq M\} \\
	& = V_{\delta_M}(f^{(t_0t_1t_2\dots)}).
\end{split}
\end{equation*}
Note that by Lemma~\ref{LipschitzTubes} {\bf (b)}, $V_{\delta_M}(f^{(t_0t_1t_2\dots)})$ is a union of $1/\delta_M$ balls of radius $\delta_M$, and so its Haar measure is $(1/\delta_M)\delta_M^2=\delta_M=1/q^{M+1}=\mu_T(U_j)$.  We have shown that
$$
\lambda(\bigcup_{(x,y)\in U_j}W_{x,y})=\mu_T(U_j)
$$
and we obtain
\begin{equation*}
\begin{split}
\lambda(E) & \leq \sum_j\lambda(\bigcup_{(x,y)\in U_j}W_{x,y}) =\sum_j\mu_T(U_j)\leq\epsilon.
\end{split}
\end{equation*}
As $\epsilon>0$ was arbitrary, we conclude that $E$ has Haar measure zero.

Using the fact that any two points in the same orbit have (eventually) the same forward orbit, it is not hard to see that if a point $(x,y)\in K^2$ is $\mu_T$-generic, then every point in its orbit $\{T^n(x,y)\}_{n=-\infty}^{+\infty}$ is $\mu_T$-generic.  Also, a straightforward argument using the fact that polynomial maps are Lipschitz on bounded sets implies that if $E\subseteq K^2$ has Haar measure zero, then so does $T^{-1}(E)$.  We conclude from these facts and $(\ref{BasinDef})$ that
\begin{equation*}
\begin{split}
\{(x,y)\in\Bcal_T\mid (x,y)\text{ is not $\mu_T$-generic}\} & = \bigcup_{n\geq1}T^{-n}(E)
\end{split}
\end{equation*}
has Haar measure zero.
\end{proof}

\section{The Hausdorff dimension of $\Acal_T$}

We follow the notation and definitions in \cite{MR1681462} $\S$11.2-11.3 for our discussion of Hausdorff dimension.  Given a subset $A\subseteq K^2$, let 
$$
\diam(A)=\sup\{\|(x,y)-(x',y')\| \,\mid\,(x,y),(x',y')\in A\}
$$
denote its diameter, and note that by the strong triangle inequality, the diameter of a ball in $K^2$ is the same as its radius.  

For parameters $p\geq0$ and $\delta\in q^\ZZ$, set 
$$
\Hcal_{p,\delta}(A)=\inf\sum_j\diam(B_j)^p,
$$
the infimum over all countable covers $\{B_j\}$ of $A$ with $\diam(B_j)\leq\delta$.  The {\em $p$-dimensional Hausdorff outer measure} of $A$ is defined by
\begin{equation}\label{HausOuterMeas}
\Hcal_p(A)=\lim_{\delta\to0}H_{p,\delta}(A).
\end{equation}
The {\em Hausdorff dimension} of $A$ is the unique nonnegative real number $\dim A$ satisfying
\begin{equation}\label{HausDimDef}
\begin{split}
\Hcal_p(A) & =0\text{ for all } p>\dim A, \text{ and} \\
\Hcal_p(A) & =+\infty \text{ for all } p<\dim A.
\end{split}
\end{equation}
For the existence of the limit $(\ref{HausOuterMeas})$ and of such a unique real number satisfying $(\ref{HausDimDef})$, see \cite{MR1681462} $\S$11.2.

\begin{lem}\label{BernMeasLem}
There exists a constant $C>0$ such that $\mu_T(A)\leq C\cdot\diam(A)^\alpha$ for all Borel measurable subsets $A$ of $K^2$, where $\alpha=1+1/(1+\log_q(1/|a|))$.
\end{lem}
\begin{proof}
If $\mu_T(A)=0$, there is nothing to prove, so assume $\mu_T(A)>0$.  Since $\mu_T$ is supported on $\Acal_T$, there exists a point $(x_0,y_0)\in A\cap \Acal_T$.  Letting $r=\diam(A)$, we have $A\subseteq B_r(x_0,y_0)$, and using Theorem~\ref{MainThmIntro}, we have $\omega(\t)=(x_0,y_0)$ for some $\t=(t_k)\in\FF^\ZZ$.  

In the notation of Lemmas~\ref{ForwardTrajectoryLemma} and \ref{BackwardTrajectoryLemma}, we let $n$ and $m$ be the positive integers for which $\delta_{n+1}<r\leq\delta_n$ and $\epsilon_{m+1}<r\leq\epsilon_m$, and we have
\begin{equation*}
\begin{split}
(x_0,y_0)=\omega(\t) & \in V(f^{(t_0t_1t_2\dots)})\cap H(g^{(\dots t_{-3}t_{-2}t_{-1})}) \\
	& \subseteq V_{\delta_n}(f_n^{(t_0t_1t_2\dots)})\cap H_{\epsilon_m}(g_m^{(\dots t_{-3}t_{-2}t_{-1})})
\end{split}
\end{equation*}
and so using Lemma~\ref{LipschitzTubes} {\bf (a)} we have 
\begin{equation*}
\begin{split}
B_r(x_0,y_0)\subseteq V_{\delta_n}(f_n^{(t_0t_1t_2\dots)})\cap H_{\epsilon_m}(g_m^{(\dots t_{-3}t_{-2}t_{-1})}).
\end{split}
\end{equation*}

If a point $\s=(s_k)\in\FF^\ZZ$ satisies $\omega(\s)\in V_{\delta_n}(f_n^{(t_0t_1t_2\dots)})\cap H_{\epsilon_m}(g_m^{(\dots t_{-3}t_{-2}t_{-1})})$, then $(\ref{TrajectorySetsAreTubes})$, $(\ref{HorTrajectorySetsAreTubes})$, and Theorem~\ref{MainThmIntro} imply that $s_k=t_k$ for all $-(m+1)\leq k\leq n$.  It follows from the definition of $\mu_T$ as the pushforward of $\mu_\sigma$ that the $\mu_T$-measure of $V_{\delta_n}(f_n^{(t_0t_1t_2\dots)})\cap H_{\epsilon_m}(g_m^{(\dots t_{-3}t_{-2}t_{-1})})$ is bounded above by $1/q^{m+n+2}$.  

Recall that $r>\delta_{n+1}=1/q^{n+2}$ and $r>\epsilon_{m+1}=|a|^{m+1}/q^{m+2}>(|a|/q)^{m+2}$.  We therefore have $-(m+2)<\frac{\log_qr}{1+\log_q(1/|a|)}$, and so
$$
1/q^{m+2}<r^{\frac{1}{1+\log_q(1/|a|)}}.
$$  
We conclude
\begin{equation*}
\begin{split}
\mu_T(A) & \leq \mu_T(B_r(x_0,y_0)) \\
	& \leq \mu_T(V_{\delta_n}(f_n^{(t_0t_1t_2\dots)})\cap H_{\epsilon_m}(g_m^{(\dots t_{-3}t_{-2}t_{-1})})) \\
	& \leq 1/q^{m+n+2} \\
	& < q^{2}r^{1+\frac{1}{1+\log_q(1/|a|)}}.
\end{split}
\end{equation*}
\end{proof}

\begin{proof}[Proof of Theorem~\ref{DimThmIntro}]
Let $\alpha=1+1/(1+\log_q(1/|a|))$.  The following standard argument (the mass distribution principle) uses Lemma~\ref{BernMeasLem} to show that the Hausdorff dimension of $\Acal_T$ is at least $\alpha$.  Assuming that $\{B_j\}$ is a countable cover of $\Acal_T$, we may assume without loss of generality that each $B_j$ is closed since $\diam(B_j)=\diam(\overline{B}_j)$.  We have
$$
1=\mu_T(\Acal_T)\leq\mu_T(\cup_jB_j)\leq\sum_j\mu_T(B_j)\leq C\sum_j\diam(B_j)^\alpha
$$
and therefore $\Hcal_{\alpha}(A)\geq1/C>0$.  We conclude that $\dim\Acal_T\geq\alpha$.

To prove the opposite inequality, let $n\geq1$ and set $\epsilon_n=|a|^n/q^{n+1}$.  Recall from the proof of Lemma~\ref{BackwardTrajectoryLemma} that $\Acal_T$ is covered by $q^n$ sets of the form $H_{\epsilon_{n}}(g)$.  By Lemma~\ref{LipschitzTubes} {\bf (b)}, each set $H_{\epsilon_{n}}(g)$ is covered by $1/\epsilon_n$ balls of radius $\epsilon_n$.  Thus $\Acal_T$ is covered by $q^n/\epsilon_n$ balls of radius $\epsilon_n$, and so 
\begin{equation*}
\begin{split}
\Hcal_{\alpha,\epsilon_n}(\Acal_T) & \leq (q^n/\epsilon_n)\epsilon_n^\alpha \\
	& =q^{n} (|a|^n/q^{n+1})^{\alpha-1} \\
	& =q^{n+(\alpha-1)(n\log_q|a|-n-1)}  \\
	& =q^{1-\alpha}
\end{split}
\end{equation*}
Letting $n\to+\infty$ we find $\Hcal_{\alpha}(\Acal_T)\leq q^{1-\alpha}<+\infty$ and hence $\dim\Acal_T\leq\alpha$.
\end{proof}

%\bibliographystyle{siam}

%\bibliography{CJP}

\end{document}